\documentclass[preprint]{elsarticle}
\usepackage{cmap} 
           \usepackage[T2A]{fontenc}   
           \usepackage[utf8]{inputenc}
           \usepackage[english]{babel}
\usepackage{amsmath}
\usepackage{amsfonts}
\usepackage{amssymb}
\usepackage{amsthm}
\usepackage[pdftex,unicode]{hyperref}
\usepackage{amscd}
\usepackage[all,cmtip]{xy}
\usepackage{mathtools}
\usepackage{comment}
\usepackage{tikz-cd}
\usetikzlibrary{babel}

  \def\R{\mathbb R}

\def\om{\omega}
\def\Om{\Omega}

\def\be{\beta}

\def\al{\alpha}

\def\ph{\varphi}
\def\de{\delta}

\def\nfs/{NFS}
\def\cdp/{CDP}
\def\cdpz/{CDP${}_0$}

\def\sq#1#2{(#1)_{#2}}
\def\sqn#1{\sq{#1}{n\in\om}}

\def\cl#1{\overline{#1}}

\def\id{\operatorname{id}}
\def\sp{\operatorname{sp}}
\def\bt{\beta}

\def\es{\varnothing}

\def\nom{{n\in\om}}

\def\sset#1{\{#1\}}

\def\set#1{\bbset#1\eeset}
\def\bbset#1:#2\eeset{\{#1\,:\,#2\}}

\def\bbsett#1:#2\eesett{\{#1\,:\,\text{#2}\}}

\def\ibbset#1:#2\ieeset{(#1)_{#2}}

\def\gm{{\mathfrak{M}}}

\def\DD#1#2#3{\D_{#1,#2}(#3)}

\def\cP{{\mathcal P}}

\def\gF{{\mathfrak{F}}}

\def\cF{{\mathcal F}}
\def\cS{{\mathcal S}}

\def\gi{{\mathfrak{i}}}
\def\gm{{\mathfrak{m}}}

\def\cV{{\mathcal V}}

\def\diag{\mathop{\bigtriangleup}}

\newcommand\restrA[2]{{
  \left.\kern-\nulldelimiterspace 
  #1 
  \vphantom{\big|} 
  \right|_{#2} 
  }}

\newcommand\restrB[2]{\ensuremath{\left.#1\right|_{#2}}}

\def\restr#1#2{\restrB{#1}{#2}}

\def\pwr#1_#2{#1^{[#2]}}

\def\te{\theta}

\def\D{\Delta}
\def\term#1{{\it #1}}

\def\dddm#1(#2){N_{#1}(#2)}
\def\dddb#1(#2){B_{#1}(#2)}

\def\wh#1{\widehat{#1}}
\def\et(#1){ (#1)}

\def\bitem#1,#2.{ $#2\nrightarrow #1$:\ }

\newtheorem{proposition}{Proposition}
\newtheorem{theorem}{Theorem}

\newtheorem*{lemma*}{Lemma}
\newtheorem{cor}{Corollary}

\theoremstyle{definition}
\newtheorem{definition}{Definition}
\newtheorem{question}{Problem}

\theoremstyle{remark}

\newtheorem*{note*}{Note}

\def\oo#1/{$O_{#1}$}

\def\gd/{$G_\delta$}

\def\F{\Phi}

\def\rarr{\Rightarrow}

\def\lrarr{\Leftrightarrow}

\def\gm{{\mathfrak{m}}}
\def\bx{\bar x}
\def\wh#1{\widehat{#1}}

\def\DD{\operatornamewithlimits{%
  \mathchoice{\vcenter{\hbox{\huge \Delta}}}
             {\vcenter{\hbox{\Large \Delta}}}
             {\mathrm{\Delta}}
             {\mathrm{\Delta}}}}
             
\def\DD{\operatornamewithlimits{\mbox{\hbox{\huge \Delta}}}}

\def\DD{\operatornamewithlimits{\Delta}}

\def\DD{\operatornamewithlimits{%
  \mathchoice{\vcenter{\hbox{\huge $\Delta$}}}
             {\vcenter{\hbox{\Large $\Delta$}}}
             {\mathrm{\Delta}}
             {\mathrm{\Delta}}}}

\def\DD{\operatornamewithlimits{%
  \mathchoice{\vcenter{\hbox{\huge $\bigtriangleup$}}}
             {\vcenter{\hbox{\Large $\bigtriangleup$}}}
             {\mathrm{\bigtriangleup}}
             {\mathrm{\bigtriangleup}}}}

\def\DD{\operatornamewithlimits{%
  \mathchoice{\vcenter{\hbox{\Large $\bigtriangleup$}}}
             {\vcenter{\hbox{\large $\bigtriangleup$}}}
             {\mathrm{\bigtriangleup}}
             {\mathrm{\bigtriangleup}}}}

\def\diag{\DD}
\def\diagsmall{\bigtriangleup}

\def\si{\sigma}

\def\nom{{n\in\om}}

\def\gF{{\mathfrak{F}}}

\def\vt#1{\tilde{#1}}

\def\rlim{\varprojlim}

\def\op#1#2{{\mathfrak{e}}^{#1}_{#2}}

\def\AEx{\mathrm{AE}}

\def\:{\colon}

\begin{document}

\begin{frontmatter}

\title{Extensions and factorizations of topological and semitopological universal algebras}
\author{Evgenii Reznichenko} 

\ead{erezn@inbox.ru}

\address{Department of General Topology and Geometry, Mechanics and  Mathematics Faculty, M.~V.~Lomonosov Moscow State University, Leninskie Gory 1, Moscow, 19991 Russia}

\begin{abstract}
The possibility of extending operations of topological and semitopological algebras to their Stone–Čech compactification and factorization of continuous functions through homomorphisms to metrizable algebras are investigated. Most attention is paid to pseudocompact and compact algebras.
\end{abstract}
\begin{keyword}
universal algebra
\sep
topological algebra
\sep
semitopological algebra
\sep
Mal'cev algebra
\sep
topological group
\sep
semitopological group
\sep
paratopological group
\sep
pseudocompact algebra
\sep
compact algebra
\end{keyword}
\end{frontmatter}

\def\vars{\operatorname{vars}}

\def\fV{{\mathfrak{V}}}
\def\ppfont#1{\mathsf{#1}}

\def\Tmp{\ppfont{T}^{+}}

\def\Tm{\ppfont{T}}
\def\Tf{\ppfont{F}}
\def\FTm{\ppfont{T_f}}
\def\Plm{\ppfont{P}}
\def\FPlm{\ppfont{P_f}}

\def\FTmn#1{\ppfont{T_{f,\mathnormal{#1}}}}
\def\FPlmn#1{\ppfont{P_{f,\mathnormal{#1}}}}

\def\Tr{\ppfont{Tr}}
\def\FTr#1{\ppfont{Tr_{f,\mathnormal{#1}}}}

\def\vv#1{\mathsf{v_{#1}}}
\def\vV{{{\mathbb V}}}

\def\alg#1{\mathbf{#1}}

\def\simf#1{\operatorname{\sim_{#1}}}

\def\itm#1#2{{\rm{(${#1}_{#2}$)}}}

\def\mpc/{$\mu^\#_{pc}$}

\def\sp{\operatorname{Sp}}

\def\Q{\mathbb Q}
\def\Z{\mathbb Z}

\section{Introduction}
\label{sec:intro}
A \term{universal algebra} or, briefly, \term{algebra}  is a set $X$ with a set of operations; an $n$-ary \term{operation} is a mapping from $X^n$ to $X$.
If $X$ has a topology and the operations on $X$ are continuous, then $X$ is a \term{topological algebra}\footnote{This is the definition for algebras with a finite number of operations. The general definition is given in Section \ref{sec:prel:topalg}.}.
If the operations on $X$ are separately continuous, then $X$ is called a \term{semitopological algebra}.

The paper examines the possibility of extending the operations of a topological algebra to its Stone–Čech compactification and factorization of continuous functions through homomorphisms to metrizable algebras.
Most attention is paid to pseudocompact and compact algebras.

The main results are formulated in Section \ref{sec:main}.

By a space we mean a Tychonoff space, unless a separation axiom  is explicitly stated.

\section{Definitions and background information}
\label{sec:prel}

Terminology and notation on universal algebra can be found in \cite{BurrisSankappanavar1981,Bergman2011}.
Topological and semitopological algebras with continuous signature were considered in \cite{ChobanDumitrashku1981,Choban1986,ChobanChiriac2013}.
Topological terminology and notation are the same as in \cite{EngelkingBookGT,HandbookOfTopology1984}.

Non-negative integers are denoted as $\om$.

\subsection{Algebras}
\label{sec:prel:alg}

A set $E$ being the union of pairwise disjoint sets $E_n$ for $\nom$ is called \term{signature}.
A signature $E$ is used for indexing algebra operations and $E_n$ is used for indexing $n$-ary operations.
Some $E_n$ may be empty. We denote
\[
\sp E = \set{\nom: E_n\neq\es}.
\]
The signature $E$ together with the function $E\to \om$ under  which the preimage of each $n\in\om$ is the set $E_n$ is also called \term{type} or \term{language}.

Let $X$ be a set, and let $\op Xn\: E_n \times X^n \to X$ for $\nom$. The pair $(X,\sqn{\op Xn})$ is called an  \term{algebra with signature $E$} or an \term{$E$-algebra}. The mappings $\op Xn$ are called the \term{algebraical structure} of the algebra $X$.
Each symbol $t\in E_n$ is associated with the $n$-ary operation defined as
\[
t^X \: X^n \to X,\ (x_1,x_2,...,x_n) \mapsto \op Xn(t,x_1,x_2,...,x_n).
\]
The zero degree $X^0$ of a set $X$ is the singleton $\sset\es$.
Identifying $X\times{\sset\es}$ with $X$, we assume that the mapping $\op X0$ maps $E_0$ to $X$ and, for  each $k\in E_0$, associate the nullary operation $k^X$  with an element of $X$, that is, $k^X\in X$.
We call nullary operations \term{constants}.

The mapping $\ph$ from an $E$-algebra $X$ to an $E$-algebra $Y$ is called a \term{homomorphism} if
\[
\ph(t^X(x_1,x_2,...,x_n))=t^Y(\ph(x_1),\ph(x_2),...,\ph(x_n))
\]
for $\nom$,   $t\in E_n$, and $x_1,x_2,...,x_n\in X$. A subset $Y\subset X$ is called a \term{subalgebra} of $X$ if $t^X(x_1,x_2,...,x_n)\in Y$
for $\nom$,   $t\in E_n$, and $x_1,x_2,...,x_n\in Y$.
The subalgebra $Y$ has a natural algebraic structure: $\op Yn=\restr{\op Xn}{E_n\times Y^n}$ for all $\nom$.
An equivalence relation $\sim$ on $X$ is called a \term{congruence} if
\[
t^X(x_1,x_2,...,x_n) \sim t^X(y_1,y_2,...,y_n)
\]
for any $\nom$, $t\in E_n$, and  $x_1,x_2,...,x_n,y_1,y_2,...,y_n\in X$ such that $x_i\sim y_i$ for $i=1,2,...,n$.
Given an equivalence relation $\sim$  on $X$, on the quotient set $X/{\sim}$ there is a (unique) $E$-algebra structure with respect to which the quotient map $X\to X/{\sim}$ is a homomorphism if and only if $\sim$ is a congruence.
A surjective mapping $\ph\: X\to Y$ corresponds to the equivalence relation $\simf\ph$ on $X$ defined by setting $x\simf\ph y$ if $\ph(x)=\ph(y)$.
A surjective map $\ph\: X\to Y$ is a homomorphism if and only if $\simf\ph$ is a congruence and the bijective map $X/{\simf\ph}\to Y$ is a homomorphism.

When this does not cause confusion, we sometimes  omit the superscript $X$.
If a signature $E$ is finite and $E=\sset{t_1,t_2,...,t_m}$, then the $E$-algebra $X$ will also be denoted as $(X,t_1,t_2,... ,t_m)$. 

Many algebraic objects are algebras, for example, groups. There are three operations on a group 
$G$: 
a binary operation $\gm\: G\times G\to G,\ (g,h)\mapsto gh$ (multiplication); 
a unary operation $\gi\: G\to G,\ g\mapsto g^{-1}$  (inversion); 
and a nullary operation $e$ (the identity element of $G$).
Thus, the group $G$ is an algebra $(G,\gm,\gi,e)$.

A ternary operation $\mu\:X^3\to X$ on a set $X$ is called a \term{Mal'cev operation} if
$\mu$ satisfies the identity
\begin{equation}\label{eq:malc}
 \mu(y,y,x)\approx\mu(x,y,y)\approx x,
\end{equation}
that is, $\mu(y,y,x)=\mu(x,y,y)=x$ for $x,y\in X$.

We call an algebra $(X,\mu)$ a \term{universal Mal'cev algebra} or, briefly, a \term{Mal'cev algebra}. 
\footnote{The term `Mal'cev algebra' is also used to refer to some linear algebras over a field.}

\subsection{Topological and semitopological algebras}
\label{sec:prel:topalg}

Suppose that a signature $E$  is endowed with a topology with respect to which the sets $E_n$ are clopen. We call a signature with  such a topology \term{continuous}. Let $X$ be an $E$-algebra with a topology.
The algebra $X$ is called a \term{(semi)topological algebra} if the mappings $\op Xn$ are (separately) continuous \cite{choban1992,Choban1993,ChobanChiriac2013}.

Suppose that $E$ is a discrete space. Then $X$ is a topological algebra if and only if $t^X\: X^n\to X$ is a continuous mapping for all $n\in\om$ and any symbol $t\in E_n$. This is how a topological algebra was defined in the works \cite{malcev1954} of Mal'cev  and \cite{Taylor1977} of Taylor. We can say that  topological algebras in the sense of Mal'cev and Taylor are topological algebras with a discrete signature $E$ in the sense of Choban \cite{ChobanChiriac2013}.
An algebra $X$ with a topology is a semitopological algebra if and only if $t^X\: X^n\to X$ is a separately continuous mapping for each $n\in\om$ and any symbol $t\in E_n$. In this way, a quasi-topological algebra was defined in \cite{SipachevaSolonkov2023}.

A group $G=(G,\gm,\gi,e)$ with a topology is called
\begin{itemize}
\item
\term{topological} if the mappings $\gm$ and $\gi$ are continuous, that is, the algebra $(G,\gm,\gi,e)$ is topological;
\item
\term{paratopological} if the mapping $\gm$ is continuous, that is, the algebra $(G,\gm,e)$ is topological;
\item
\term{quasitopological} if the mapping $\gm$ is separately continuous and $\gi$ is continuous, that is, the algebra $(G,\gm,\gi,e)$ is semitopological;
\item
\term{semitopological} if the mapping $\gm$ is separately continuous, that is, the algebra $(G,\gm,e)$ is semitopological.
\end{itemize}

Let $X$ be a topological space with a Mal'cev operation $\mu$. The algebra $(X,\mu)$ is a topological Mal'cev algebra if $\mu$ is continuous. 
A topological space is a \term{Mal’cev space} if it admits a continuous Mal’cev operation, that is, if $X$ is homeomorphic to some  topological Mal’cev algebra. Any group $G$ has a natural Mal'cev operation: $\mu(x,y,z)=xy^{-1}z$. If $r\: G\to G$ is a retraction, then there is a natural Mal'cev operation on $X=r(G)$: $\mu(x,y,z)=r(xy^{-1}z)$.
Consequently, a retract of a topological group is a  Mal'cev space.
A pseudocompact Mal'cev  space is a retract of a topological group \cite{Sipacheva1991,ReznichenkoUspenskij1998}. There exists a completely regular Mal'cev space which is not a retract of a topological group \cite{grs1997}.

The algebra $(X,\mu)$ is a semitopological Mal'cev algebra if $\mu$ is separately continuous.
A topological space is a \term{quasi-Mal’cev} space if it admits a separately continuous Mal’cev operation, that is, if $X$ is homeomorphic to some semitopological Mal’cev algebra. 
A space is quasi-Mal’cev if and only if it is a retract of a quasitopological group \cite{SipachevaSolonkov2023}.

The most interesting and studied algebras are those with finite signature. 
With the help of algebras with an infinite signature it is possible to describe more algebraic structures, for example, vector spaces. Using algebras with a continuous signature, we can describe topological vector spaces.

\subsection{Terms, identities, and equational classes}

Let $V$ be a \term {set of variables} that does not intersect the signature $E$.
The set $\Tm(E,V)$ of \term{terms of signature $E$ over $V$} is the smallest set such that it contains $V\cup E_0$
and if $m\in\om$, $f\in E_m$, and $t_1,t_2,...,t_m\in \Tm(E,V)$, then $(f,t_1,t_2,... ,t_m)\in \Tm(E,V)$.
A term $(f,t_1,t_2,...,t_m)$ is also written as $f(t_1,t_2,...,t_m)$.
For binary operations, such as the $+$ operation, the term $(+,x,y)$ can be written as $+(x,y)$ and $x+y$.

For a term $t$, let $\vars t$ denote the \term {set of variables} in $t$.
For $v\in V$, we denote by $\Om(t)(v)$ \term {the number of occurrences of the variable} $v$ in $t$.
If $X$ is an $E$-algebra, $\vars t\subset R\subset V$, and $\bar x=\sq{x_v}{v\in R} \in X^R$, then we define $t^X(\bar x)\in X$ as the result of substituting the values $x_v$ into $t$ instead of the variables $v\in \vars t$ and then calculating the obtained expression. Let us define the introduced concepts more formally, by induction.

Let $t\in\Tm(E,V)$, $v\in V$, $\vars t\subset R\subset V$, and $\bar x=\sq{x_v}{v\in R}\in X^R$.

We put
\begin{description}
\item[$t\in E_0$:\ ]
$\vars t=\es$, $\Om(t)(v)=0$, and $t^X(\bar x)=\op X0(t)=t^X$;
\item[$t\in V$:\ ] $\vars t=\sset t$,
\[
\Om(t)(v) = \begin{cases}
1,& t=v,
\\
0,& t\neq v,
\end{cases}
\]
and $t^X(\bar x)=x_t$;
\item[$t\notin E_0\cup V$:\ ] if $m\in\om$, $f\in E_m$ $t_1,t_2,...,t_m\in \Tm(E,V)$, and $t=f(t_1,t_2,...,t_m)$, then
\begin{align*}
\vars t&=\bigcup_{i=1}^m \vars t_i,
\quad \Om(t)(v)=\sum_{i=1}^m \Om(t_i)(v),
\\
t^X(\bar x)&=\op Xm(f,t_1^X(\bar x),t_2^X(\bar x),...,t_m^X(\bar x))=f^X(t_1^X(\bar x),t_2^X(\bar x),...,t_m^X(\bar x)).
\end{align*}
\end{description}

Let $\fV_E$ denote the category of all $E$-algebras without topological structure.
Morphisms in this category are homomorphisms.
We denote the category of all topological (semitopological) $E$-algebras by $\fV_E^T$ ($\fV_E^{ST}$).
Morphisms in $\fV_E^T$ and $\fV_E^{ST}$  are continuous homomorphisms.

A subcategory $\cP$ of the category $\fV_E$ is called a \term{variety} if $\cP$ is closed under products, subalgebras, and homomorphic images.

A pair of terms $p,q\in \Tm(E,V)$, written as $p\approx q$, is called an \term{identity}. They say that \term{an identity $p\approx q$ holds for an $E$-algebra $X$} if $p^X(\bar x)=q^X(\bar x)$ for all $\bar x\in X^R$, where $R={\vars p\cup \vars q}$.

Let us fix a countable set
\[
\vV=\sset{\vv 1,\vv 2, ...}
\]
of variables. We denote $\Tm(E)=\Tm(E,\vV)$.

Let $\Xi\subset \Tm(E)\times \Tm(E)$ be a set of identities.
We denote
\begin{align*}
\fV_E[\Xi]=\{ X\in \fV_E:& \text{ for all } p\approx q\in\Xi
\\
&\text{the identity }p\approx q\text{ holds for } X\}.
\end{align*}
We also denote by $\fV_E^T[\Xi]$ ($\fV_E^{ST}[\Xi]$)  the class of topological (semitopological) $E$-algebras in $\fV_E[\Xi]$.
The class of $E$-algebras of the form $\fV_E[\Xi]$ is called an \term{equational class}.\footnote{Such classes are also called primitive classes.}.
Birkoff's theorem \cite[Theorem 11.9]{BurrisSankappanavar1981} states that a subcategory $\cP$ of the category of all $E$-algebras $\fV_E$ is a variety if and only if $\cP$ is an equational class.

For example, groups with signature $\{e,{}^{-1},\cdot\}=\sset{e,\gi,\gm}$ are defined by the identities $x e\approx e x\approx x$, $x x^{-1}\approx  x^{-1}x\approx e $, and $x(yz)\approx (xy)z$.  Abelian groups in additive notation with signature\footnote{Here `$-$' is the unary operation of taking the opposite element} $\{0,-,+\}$   are defined by the identities $0+x\approx x+0=x$ , $x+(-x)\approx (-x)+x\approx 0$, $(x+y)+z\approx x+(y+z)$, and $x+y\approx y+x$.

\subsection{Translation semigroup}
\label{sec:prel:transl}

Let $X$ be an $E$-algebra.
For $n\in\sp E$, $n>0$, $i\in\sset{1,2,...,n}$, $f\in E_n$, and 
\[
\bar x=(x_1, x_2,...,x_{i-1},x_{i+1},...,x_n)\in X^{n-1},
\]
we denote
\[
T_{f,i,\bar x}\: X\to X,\ x\mapsto f^X(x_1,x_2,...,x_{i-1},x,x_{i+1},. ..,x_n).
\]
Mappings of the form $T_{f,i,\bar x}$ are called \term{principal translations} (or \term{elementary translations}).
Let $S_1(X)$ denote the set of all principal translation  in an $E$-algebra $X$.
Let $e$ denote the identity mapping $\id_X$ of the space $X$ onto itself,
and let $S(X)$ be the subsemigroup in $X^X$ generated by $S_1(X)\cup \sset e$. The elements of the semigroup $S(X)$ are called \term{translations}. Any translation is a composition of principal translations.

\begin{proposition}[{\cite[Theorem 4.16]{Bergman2011}}]\label{p:prel:1}
Let $X$ be an $E$-algebra. An equivalence relation $\sim$ on $X$ is a congruence if and only if the following condition is satisfied: if $x\sim y$ and $\si\in S_1(X)$, then $\si(x)\sim \si(y)$.
\end{proposition}


\section{Main results}
\label{sec:main}

\subsection{Extension of operations on $\be X$}

We are interested in when (separately) continuous operations $\Psi$ of an $E$-algebra $X$ extend to (separately) continuous operations $\sqn{\op Xn}$ on $\bt X$, that is, when the mapping $\op Xn$ extend to 
a (separately) continuous mapping $\op{\bt X}n\: E_n\times (\bt X)^n\to \bt X$ for $n\in\sp E$.

Dugundji compacta were introduced by Pe\l czy\'nski in \cite{Pelczynski1968}. 
A compact space $X$ belongs to the class $\AEx(0)$ of \term{absolute extensors in dimension
zero} if,  given a closed subset $Y$ of a Cantor cube $\sset{0,1}^A$, every
continuous mapping from $Y$ to $X$ extends to a continuous mapping from $\sset{0,1}^A$ to
$X$. In \cite{Haydon1974}, Haydon proved that the compact spaces of class $\AEx(0)$ are exactly the Dugundji compacta.
Often compact spaces encountered in  topological algebra are Dugundji compacta \cite{Uspenskii1989}.

\begin{theorem}\label{t:main:1} Let $G$ be a group with a topology. The following conditions are equivalent:
\begin{enumerate}
\item
$G$ is a topological group and the operations of $G$ extend to continuous group operations on $\bt G$, that is, $\bt G$ is a topological group and $G$ is embedded in $\bt G$  as a subgroup;
\item
$G$ is a pseudocompact topological group;
\item
$G$ is a topological group and $G^\tau$ is pseudocompact for any cardinal $\tau$;
\item
$G$ is a topological group and $\bt G$ is a Dugundji compactum;
\item
$G$ is a pseudocompact paratopological group;
\item
$G$ is a pseudocompact semitopological group and the multiplication of $G$ extends to a separately continuous map on $\bt G\times \bt G$.
\end{enumerate}
\end{theorem}
\begin{proof}
The  equivalence of the first three conditions is the Comfort--Ross theorem \cite[Theorem 1.4, Theorem 4.1]{ComfortRoss1966}.
If $\bt G$ is a Dugundji compactum, then $G$ is pseudocompact \cite[Proposition 4]{Uspenskii1989},  and hence (4) $\rarr$ (3). Compact topological groups are Dugundji compacta \cite{Pasynkov1976,Uspenskii1989}, and hence (1) $\rarr$ (4).
The equivalences of (2) $\lrarr$ (5) $\lrarr$ (6) were proved in \cite{Reznichenko1994}.
\end{proof}

\begin{question}\label{q:main:1}
Let $G$ be a topological group such that the  multiplication of $G$ extends to a separately continuous operation on $\be G$. Is it true that $G$ is pseudocompact?
\end{question}

Let ($P_1$) be the condition of Problem \ref{q:main:1}.
Note that the following conjectures are equivalent:
\begin{itemize}
\item
if ($P_1$), then $G$ is pseudocompact;
\item
if ($P_1$), then $\be G$ is a topological group;
\item
if $G$ is a semitopological group and multiplication extends to a separately continuous operation on $\be G$, then $G$ is pseudocompact and $\be G$ is a topological group.
\end{itemize}

\begin{theorem}\label{t:main:2} Let $X$ be a topological Mal'cev algebra. The following conditions are equivalent:
\begin{enumerate}
\item
the Mal'cev operation on $X$ extends to a continuous Mal'cev operation on $\bt X$, that is, there is the structure of a topological Mal'cev algebra on $\bt X$ and $X$ is a subalgebra of $\bt X$;
\item
The Mal'cev operation on $X$ extends to a continuous operation on $\bt X$;
\item
$\bt X$ is a Dugundji compactum;
\item
$X$ is pseudocompact;
\item
$X^\tau$ is pseudocompact for any cardinal $\tau$.
\end{enumerate}
\end{theorem}
\begin{proof}
In \cite{ReznichenkoUspenskij1998} it was proved that (3) $\lrarr$ (4) $\lrarr$ (5) $\rarr$ (1). Since a compact topological Mal'cev algebra is a Dugundji compactum \cite[Corollary 6]{Uspenskii1989}, it follows that (1) $\rarr$ (3).
Obviously, (1) $\rarr$ (2).
It is easy to see \cite[Proposition 4.4]{Reznichenko2024tp} that if in a topological algebra with a ternary operation some dense subalgebra is Mal'cev, then the entire algebra is Mal'cev. Therefore, (2) $\rarr$ (1).
\end{proof}

\begin{question}\label{q:main:2}
Let $X$ be a topological Mal'cev algebra, and let the Mal'cev operation extend to a separately continuous operation on $\be X$. Is it true that $X$ is pseudocompact?
\end{question}

Let ($P_2$) be the condition of Problem \ref{q:main:2}.
Note that
the following conjectures are equivalent:
\begin{itemize}
\item
if ($P_2$), then $X$ is pseudocompact;
\item
if ($P_2$), then $\bt X$ is Dugundji;
\item
if ($P_2$), then the operation on $\be X$ is a continuous Mal'cev operation, $X$ is pseudocompact, and $\bt X$ is Dugundji.
\end{itemize}

A mapping $f \: X \to Y$ is called quasicontinuous if for every point $x \in X$, any neighborhood $O$ of $f(x)$, and any neighborhood $W$ of $x$, there exists a non-empty open $U \subset W$ such that $f(U)\subset O$.

Suppose that $\set{X_\al: \al\in A}$ is a family of sets, $Y$ is a set, $X=\prod_{\al\in A} X_\al$, $\F\: X \to Y$ is a mapping, $B\subset A$, and $\bx = \sq{x_\al}{\al \in A\setminus B}\in \prod_{\al\in A\setminus B} X_\al$. Let us define the mapping
\[
r(\F,X,\bx)\: \prod_{\al\in B} X_\al\to Y,\ \sq{x_\al}{\al \in B}\mapsto \F(\sq{ x_\al}{\al \in A}).
\]

\begin {definition} \label{d:fsv:1}
Given a set $A$  and a family of spaces $ \set {X_ \al: \al \in A} $, let $ X = \prod _{\al \in A}X_\al$, and let $Y$ be a space.
%
Suppose given a map $ \F \: X \to Y $ and a positive integer $ n $.
\begin{itemize}
\item
  The map $ \F $ is {\it $ n $-separately continuous} iff $ r (\F , X, \bx)$ is continuous for each $ B \subset A $ with $ | B | \le n $ and any $ \bx \in \prod _ {\al \in A \setminus B} X_ \al $ \cite[Definition 3.25]{ReznichenkoUspenskij1998}.
\item
The map $ \F $ is {\it $ n $-$ \be $-extendable} iff $ g = r (\F , X, \bx)$ extends to a separately continuous map $ \hat g: \prod _ {\al \in B} \be X_ \al \to \be Y $ for each $ B \subset A $ with $ | B | \le n $ and any $ \bx \in \prod _ {\al \in A \setminus B} X_ \al $ \cite[Definition 1]{Reznichenko2022}.
\item
The map $ \F $ is {\it $n$-quasicontinuous} iff $ r (\F , X, \bx)$ is quasicontinuous for each $ B \subset A $ with $ | B | \le n $ and any $ \bx \in \prod _ {\al \in A \setminus B} X_ \al $ \cite[Definition 1]{Reznichenko2022-2}.
\end{itemize}
\end{definition}

Separately continuous maps are exactly $1$-separately continuous maps.

\begin{theorem}\label{t:main:3} Let $(X,M)$ be a semitopological Mal'cev algebra.
The following conditions are equivalent:
\begin{enumerate}
\item
the Mal'cev operation $M$ on $X$ extends to a separately continuous Mal'cev operation on $\bt X$, that is,  there is the structure of a semitopological Mal'cev algebra on $\bt X$ and $X$ is a subalgebra of $\bt X$;
\item
the space $X$ is pseudocompact and the Mal'cev operation on $X$ extends to a separately continuous operation on $\bt X$;
\item
the space $X$ is pseudocompact and the Mal'cev operation $M$ is $2$-$\be$-extendable;
\item
the space $X$ is pseudocompact and the Mal'cev operation $M$ is $2$-quasicontinuous.
\end{enumerate}
If {\rm (1)} holds, then $\bt X$ is a Dugundji compactum and $X^\tau$ is pseudocompact for any $\tau$.
\end{theorem}
\begin{proof}
A compact semitopological Mal'cev algebra is a Dugundji compactum \cite{ReznichenkoUspenskij1998,Reznichenko2022-2}. Consequently, from (1) it follows that $\bt X$ is a Dugundji compactum and $X^\tau$ is pseudocompact for any $\tau$ \cite{Uspenskii1989,ReznichenkoUspenskij1998}.
Then from \cite[Theorem 5]{Reznichenko2022-2} it follows that (2) $\lrarr$ (3) $\lrarr$ (4) and
from \cite[Theorem 8]{Reznichenko2022-2} it follows that (1) $\lrarr$ (2).
\end{proof}

\begin{question}\label{q:main:3}
Let $X$ be a semitopological Mal'cev algebra such that the Mal'cev operation extends to a separately continuous operation on $\be X$. Is it true that $X$ is pseudocompact?
\end{question}

Let ($P_3$) be the condition of Problem \ref{q:main:3}.
Note that
the following conjectures are equivalent:
\begin{itemize}
\item
if ($P_3$), then $X$ is pseudocompact;
\item
if ($P_3$), then $\bt X$ is Dugundji;
\item
if ($P_3$), then the operation on $\be X$ is a separately continuous Mal'cev operation, $X$ is pseudocompact, and $\bt X$ is Dugundji.
\end{itemize}

Let $X$ and $Y$ be spaces. We say that,  the pair $(X, Y)$ is \term{Grotendieck} if for every continuous map $f \: X \to C_p(Y)$ the closure of $f(X)$ in $C_p(Y)$ is compact \cite{Reznichenko1994}.
Here $C_p(X)$ denotes the space of continuous real-valued functions on $X$ with the topology of pointwise convergence. We say that a space $X$ is a \term{Korovin} space if $(X,X)$ is Grotendieck \cite{ReznichenkoTkachenko2024}.
Since $C_p(X)$ contains a closed copy of the real line, any Korovin space is pseudocompact.
A space $Y$ is called \term{\mpc/-complete} \cite{Reznichenko2024gr} or \term{weak $pc$-Grothendieck} \cite{Arhangelskii1997} if $(X,Y)$ is Grothendieck for any pseudocompact $X$.

A space $X$ is \term{Dieudonn\'e complete} if it admits a compatible complete uniformity.
For a space $X$ the \term{Dieudonn\' e completion} $\mu X$ can be defined as the smallest Dieudonn\' e
complete subspace of $\bt X$ containing $X$. If $X$ is pseudocompact, then $\mu X=\bt X$.

\begin{proposition}\label{p:main:1} Let $X$ and $Y$ be spaces.
\begin{enumerate}
\item
For $X$ and $Y$, the following conditions are equivalent:
\begin{enumerate}
\item
$(X,Y)$ is Grotendieck;
\item
every separately continuous function $\F\: X\times Y\to \R$ has a separately continuous extension $\wh\F\: \bt X\times Y\to \R$;
\item
for any space $Z$, 
every separately continuous map $\F\: X\times Y\to \R$ has a separately continuous extension $\wh\F\: \bt X\times Y\to \mu Z$.
\end{enumerate}
\item
For any space $X$, the following conditions are equivalent:
\begin{enumerate}
\item
$X$ is Korovin;
\item
every separately continuous function $\F\: X\times X\to \R$ has a separately continuous extension $\wh\F\: \bt X\times \bt X\to \R$;
\item
for any space $Z$, every separately continuous mapping $\F\: X\times X\to Z$ has a separately continuous extension $\wh\F\: \bt X\times \bt X\to \mu Z$.
\end{enumerate}
\item
Any \mpc/-complete space is Korovin.
Each of the following properties implies that $X$ is \mpc/-complete:
$X$ is $k$-space; $X$ is separable; $X$ has countable tightness; $X$ is countably compact.
\end{enumerate}
\end{proposition}
\begin{proof}
Assertion   (1) follows from \cite[Assertion 1.2]{Reznichenko1994}.
Assertion   (2) follows from \cite[Proposition 3.8]{ReznichenkoUspenskij1998}.
Assertion   (3) follows from \cite[Corollary 1.9]{Reznichenko1994}.
\end{proof}

The following statement follows from Theorem \ref{t:main:2} and Proposition \ref{p:main:1}.
 
\begin{theorem}\label{t:main:4} Let $X$ be a semitopological Korovin Mal'cev algebra.
Then the Mal'cev operation on $X$ extends to a separately continuous Mal'cev operation on $\bt X$.
\end{theorem}

Let $X$ and $Y$ be spaces. We say that $(X, Y)$ is a \term{$C$-pair} if every continuous function $\F\: X\times Y\to \R$ has a continuous extension $\wh\F\: \bt X\times Y \to \R$.
Obviously, if $(X, Y)$ is a $C$-pair, then $X$ is pseudocompact.


The following statements are the main results of the work and are proved in Section \ref{sec:ext}.

\begin{theorem}\label{t:main:5}
Let $X$ be a topological $E$-algebra. If $X^n$ is pseudocompact and $(X^n,E_n)$ is a $C$-pair for any $n\in\sp E$, then  $\op Xn$ has a continuous extension $\op{\bt X}n\: E_n \times(\bt X)^n \to \bt X$ for any $n\in\sp E$. That is, $\wh X=(\bt X,\sqn{\op{\bt X}n})$ is a topological $E$-algebra and $X$ is a subalgebra of the algebra $\wh X$.
\end{theorem}

We say that a space $Y$ is \term{$C$-universal} if $(X,Y)$ is a $C$-pair for every pseudocompact $X$.
The following statement follows from Theorem \ref{t:main:5}.

\begin{theorem}\label{t:main:6}
Let $X$ be a topological $E$-algebra.
If $E$ is $C$-universal and $X^n$ is pseudocompact for any $n\in\sp E$, then
the operations of $X$ extend to continuous operations on $\bt X$ with respect to which $\bt X$ is a topological $E$-algebra.
\end{theorem}

Any $b^*_R$-space \cite{Noble1969} is $C$-universal (Proposition \ref{p:ext:4}). In particular, any $k$-space is $C$-universal (Corollary \ref{c:ext:1}).

\begin{cor}\label{c:main:1}
Let $X$ be a topological $E$-algebra.
If $E$ is a $k$-space and $X^n$ is pseudocompact for any $n\in\sp E$, then
the operations of $X$ extend to continuous operations on $\bt X$ with respect to which $\bt X$ is a topological $E$-algebra.
\end{cor}

In connection with Theorems \ref{t:main:5} and \ref{t:main:6} and Corollary \ref{c:main:1} the question arises under what conditions a product of pseudocompact spaces is pseudocompact.
This issue is discussed in \cite[Section 1.4]{AACCICTM2018}. In \cite{Novak1953,Terasaka1952} countably compact spaces are constructed whose product is not pseudocompact.
The presence of an additional algebraic structure may entail the pseudocompactness of some powers of the space.
A pseudocompact (para)topological group is pseudocompact to any power
\cite{ComfortRoss1966,Reznichenko1994} (see Theorem \ref{t:main:1}). A pseudocompact Mal'cev algebra is pseudocompact to any power \cite{ReznichenkoUspenskij1998} (see Theorem \ref{t:main:2}).
Homogeneous pseudocompact spaces $X$ and $Y$ whose product $X \times Y$ is not pseudocompact were constructed in \cite{ComfortVanMill1985,rezn2020}. However, the problem of the existence of a homogeneous pseudocompact space $X$ for which the product $X \times X$ is not pseudocompact \cite[Question 5.3]{ComfortVanMill1985} still remains open.


\begin{theorem}\label{t:main:7}
Let $X$ be a semitopological $E$-algebra.
If $X$ is pseudocompact, $(X,E)$ is Grotendieck, and $f^X\: X^n\to X$ is $2$-quasicontinuous for any $n\in \sp E$ and $ f\in E_n$, then
the operations of $X$ extend to separately continuous operations on $\bt X$ with respect to which $\bt X$ is a semitopological $E$-algebra.
\end{theorem}


\begin{theorem}\label{t:main:8}
Let $X$ be a semitopological $E$-algebra.
If $X$ is Korovin and $(X,E)$ is Grotendieck, then
the operations of $X$ extend to separately continuous operations on $\bt X$ with respect to which $\bt X$ is a semitopological $E$-algebra.
\end{theorem}

Since continuous mappings of the product of spaces are $2$-quasicontinuous, the following statement follows from Theorem \ref{t:main:8} and Proposition \ref{p:main:1}.

\begin{cor}\label{c:main:2}
Let $X$ be a topological $E$-algebra.
If $E$ is a $k$-space and $X$ is pseudocompact, then
the operations of $X$ extend to separately continuous operations on $\bt X$ with respect to which $\bt X$ is a semitopological $E$-algebra.
\end{cor}

\subsection{Identities in extensions of $E$-algebras}

If two continuous maps coincide on a dense subset, then they coincide everywhere.
Therefore, the following statement is true.

\begin{theorem}\label{t:main:9}
Let $X$ be a topological $E$-algebra and $Y$ be a dense subalgebra of $X$, and let $p,q \in\Tm(E)$. If the identity $p\approx q$ holds for $Y$, then the identity $p\approx q$ holds for $X$.
\end{theorem}

This theorem cannot be extended to semitopological $E$-algebras. Let $\Z$ be the group of integers with the discrete topology. Let $S$ denote the one-point compactification of $\Z$ by the point $\infty$. Let us extend the operations from $\Z$ to $S$: $\infty+x=x+\infty=\infty$, $-\infty=\infty$. Then $S$ is a compact semitopological algebra with signature $(0,-,+)$. The identity $x+(-x)\approx 0$ does not hold in $S$.


The following statement is easy to verify.

\begin{proposition}\label{p:main:e1}
Let $X$ and $Z$ be  spaces, and let $f,g: X^n\to Z$ be separately continuous maps. 
If $Y\subset X=\cl Y$ and $\restr{f}{Y^n}=\restr{g}{Y^n}$, then $f=g$.
\end{proposition}

\begin{theorem}\label{t:main:10}
Let $X$ be a semitopological $E$-algebra, let $Y$ be a dense subalgebra of $X$, and let $p,q \in\Tm(E)$.
Suppose that $\Om(p)(v)\leq 1$ and $\Om(q)(v)\leq 1$ for all $v\in \vV$, that is,
that each variable from $\cV$ appears in $p$ and $q$ at most once.
Then if the identity $p\approx q$ holds for $Y$, then the identity $p\approx q$ holds for $X$.
\end{theorem}
\begin{proof}
Let us put $R=\vars p \cup \vars q$, $f(\bar x)=p^X(\bar x)$, and $g(\bar x)=q^X(\bar x)$ for $\bar x\in X^R$. Then $f$ and $g$ are separately continuous functions on $X^R$ and $\restr{f}{Y^R}=\restr{g}{Y^R}$. From Proposition \ref{p:main:e1} it follows that $f=g$.
\end{proof}

From this theorem it follows that the identities of associativity and commutativity of groups are preserved in extensions. For example, $S$ is a commutative semigroup.

\begin{cor}\label{c:main:main10}
Let $G$ be a semitopological semigroup.
If $G$ is a dense subspace of a space $S$ and  multiplication in $G$ extends to a separately continuous operation on $S$, then $S$ is a semitopological semigroup with respect to this operation.
If $G$ is commutative, then the semigroup $S$ is commutative.
\end{cor}

The following statement is proved in Section \ref{sec:ident}.

\begin{theorem}\label{t:main:11}
Let $X$ be a semitopological $E$-algebra and $Y$ be a dense pseudocompact subalgebra of $X$,  and let $p, q\in\Tm(E)$.
Assume that for each $v\in \vV$ one of the  following conditions is satisfied:
\begin{enumerate}
\item
$\Om(p)(v)\leq 1$ and $\Om(q)(v)\leq 2$;
\item
$\Om(p)(v)\leq 2$ and $\Om(q)(v)\leq 1$.
\end{enumerate}
Then if the identity $p\approx q$ holds for $Y$, then the identity $p\approx q$ holds for $X$.
\end{theorem}

The following statement follows from Theorem \ref{t:main:11}.

\begin{cor}[{\cite{ReznichenkoUspenskij1998}}]\label{c:main:3}
Let $Y$ be a pseudocompact Mal'cev semitopological algebra, $X$ be an extension of $Y$, $Y\subset X=\cl Y$, and the Mal'cev operation $\mu$ on $Y$ extend to a separately continuous operation $\hat\mu$ on $X$. Then $\hat\mu$ is a Mal'cev operation.
\end{cor}

\subsection{Factorization of mappings and embedding of algebras into products of separable metrizable algebras}

Let $X$ be an $E$-algebra.
A (semi)topological algebra $X$ is called \term{$\R$-factorizable} if, for any continuous function $f: X\to \R$, there exists a separable metrizable (semi)topological algebra $Y$, a continuous function $g: Y\to\R$, and a continuous homomorphism $h: X\to Y$ such that $f=g\circ h$.

\begin{proposition}\label{p:main:2}
If $X$ is an $\R$-factorizable (semi)topological $E$-algebra, then $X$ embeds in a product of separable metrizable (semi)topological $E$-algebras.
\end{proposition}
\begin{proof}
For each continuous function $f\in C(X)$ on $X$, we fix a separable metrizable (semi)topological algebra $Y_f$, a continuous function $g_f\: Y_f\to\R $, and a continuous homomorphism $h_f\: X\to Y_f$ such that $f=g_f\circ h_f$. Then the mapping
\[
\ph=\diag_{f\in C(X)}h_f\: X\to \prod_{f\in C(X)} Y_f
\]
is a homomorphic embedding of $X$ into the $E$-algebra $\prod_{f\in C(X)} Y_f$.
\end{proof}

Let $T$, $X_1$, $X_2$, ..., $X_n$ be spaces.
We say that the product space
$\prod_{i=1}^n X_i$
is
\term{(separately) $\R$-factorizable over $T$} if, for any (separately) continuous function $f\: T\times \prod_{i=1}^n X_i\to \R$, there are separable metrizable spaces $Y_1$, $Y_2$, ..., $Y_n$, continuous mappings $h_i\: X_i\to Y_i$ for $i=1, 2, ..., n$, and a continuous function
$g\: T\times \prod_{i=1}^n Y_i\to \R$ such that  $f=g\circ h$, where $h=\id_T\times \prod_{i=1}^n h_i$.

The following statements are proved in Section \ref{sec:rfmap}.

\begin{theorem}\label{t:main:12}
Let $X$ be a topological $E$-algebra.
If the product $X^n$ is $\R$-factorizable over $E_n$ for $n\in \sp E$, then $X$ is $\R$-factorizable.
\end{theorem}

\begin{theorem}\label{t:main:13}
Let $X$ be a semitopological $E$-algebra.
If the product $X^n$ is separately $\R$-factorizable over $E_n$ for $n\in \sp E$, then $X$ is $\R$-factorizable.
\end{theorem}

A space $X$ has \term{caliber $\om_1$} if any uncountable family of nonempty open sets in $X$ has
an uncountable subfamily with nonempty intersection. 
Any separable space has caliber $\om_1$.

\begin{proposition}\label{p:main:3} Let $T$, $X_1$, $X_2$, ..., $X_n$ be spaces.
\begin{enumerate}
\item
If $T\times \prod_{i=1}^n X_i$ is a Lindelöf space, then the product $\prod_{i=1}^n X_i$ is $\R$-factorizable over $T$.
\item
If $X_i$ is a compact space with caliber $\om_1$ and the space $T$ is separable,
then the product $\prod_{i=1}^n X_i$ is separately $\R$-factorizable over $T$.
\end{enumerate}
\end{proposition}

The following statement follows from Theorem \ref{t:main:12} and Proposition \ref{p:main:3}(1).

\begin{theorem}\label{t:main:14}
Let $X$ be a topological $E$-algebra.
If the product $E_n\times X^n$ is Lindelöf for $n\in \sp E$, then $X$ is $\R$-factorizable.
\end{theorem}

Theorem \ref{t:main:14} implies that a paratopological group with  Lindelöf square is $\R$-factorizable.
The Sorgenfrey line is a Lindelöf non-$\R$-factorizable paratopological group \cite[Remark 3.22]{XieLinTkachenko2013}. Any Lindelöf topological group is $\R$-factorizable \cite{Tkachenko1989}, \cite[Theorem 8.1.6]{at2009}. There is a separable non-$\R$-factorizable group \cite{ReznichenkoSipacheva2013,Reznichenko2024rf}.

The following statement follows from  Theorem \ref{t:main:13} and Proposition \ref{p:main:3}(2).

\begin{theorem}\label{t:main:15}
Let $X$ be a semitopological $E$-algebra.
If $X$ is a compact space with caliber $\om_1$ and $E$ is separable,
then $X$ is  $\R$-factorizable.
\end{theorem}

\subsection{Least factorizations}
In this section we do not assume any  separation axioms.

Let $X$ be a semitopological $E$-algebra,  and let $f\: X\to Z$ be a continuous map. A \term{factorization} of a mapping $f$ is a triple $(g,Y,h)$, where $Y$ is a semitopological $E$-algebra, $g\: X\to Y$ is a continuous surjective homomorphism, $h\: Y\to Z$ is continuous, and $f=h\circ g$. We denote the set of all factorizations of a mapping $f$ as $\gF( X,f,Z)$. On $\gF( X,f,Z)$ consider the partial order  defined by setting $(g_1, Y_1,h_1) \prec (g_2, Y_2,h_2)$ if there is a continuous homomorphism $q\:  Y_2\to  Y_1$, such that $g_1=g_2\circ q$. The factorization $(\id_X, X,f)$ is the greatest element in $\gF( X,f,Z)$.

The following theorem is proved in Section \ref{sec:mfac}.

\begin{theorem}\label{t:main:15}
Let $X$ be a semitopological $E$-algebra, and let $f\: X\to Z$ be a continuous map.
In the set $\gF( X,f,Z)$ of all factorizations of the mapping $f$ there is a least element $(g, Y,h)$.
Moreover, $Y$ is embedded in $Z^A$ for some set $A$.
\end{theorem}

\begin{cor}\label{c:main:4}
Let $X$ be a semitopological $E$-algebra, and let $f\: X\to Z$ be a continuous map. Suppose  that $i\in\sset{0,1,2,3,3 \frac 12}$ and $Z$ is a $T_i$ space. Then there is a semitopological $E$-algebra $Y$, 
a continuous surjective homomorphism $g\:  X\to  Y$,  and a continuous map $h\: Y\to Z$ such that $f=h\circ g $ and $Y$ is a $T_i$ space.
\end{cor}


\section{Extension of operations}
\label{sec:ext}

\begin{proposition}\label{p:ext:1} Let $X$, $Y$, and $Z$ be spaces and $(X, Y)$ be a $C$-pair.
Then any continuous map $f: X\times Y\to Z$ extends to a continuous map $\hat f: \bt X\times Y\to \bt Z$.
\end{proposition}
\begin{proof}
Let $\de\: \bt Z\to [0,1]^A$ be a topological embedding, let $\pi_\al\: [0,1]^A\to [0,1]$ be the projection onto the $\al$th coordinate, and let $\de_\al=\pi_\al \circ \de$ for $\al\in A$. Let $\hat \de_\al\: \bt X\times Y\to [0,1]$ be a continuous extension of $\de_\al$. We set $h=\diag_{\al\in A}\hat \de_\al$ and $\hat f= \de^{-1}\circ h$. Then $\hat f$ is a continuous extension of $f$.
\end{proof}

\begin{proof}[Proof of Theorem \ref{t:main:5}]
Let $n\in\sp E$. Since $(X^n,E_n)$ is a $C$-pair, it follows from Proposition \ref{p:ext:1}  that the map $\op Xn$ extends to a continuous map from $\bt(X^n)\times E_n$ to $\bt X$. Since $X^n$ is pseudocompact, it follows from the Glicksberg theorem \cite{gli1959} that $\bt(X^n)$ is homeomorphic to $(\bt X)^n$. Consequently, the mapping $\op Xn$ extends to a continuous mapping $\op {\bt X}n: E_n \times(\bt X)^n \to \bt X$.
\end{proof}

A mapping $f\: Z\to Y$ is called \term{$z$-closed} if it maps zero sets to closed sets.
Let $C^*(X)$ denote  the set of continuous bounded functions on a space $X$.

\begin{proposition}\label{p:ext:2} Let $X$ and $Y$ be spaces.
The following conditions are equivalent:
\begin{enumerate}
\item
$(X, Y)$ is a $C$-pair;
\item
$X$ is pseudocompact and the projection $\pi\: X\times Y\to Y$ is $z$-closed.
\end{enumerate}
\end{proposition}
\begin{proof}
(1) $\rarr$ (2)
Obviously, $X$ is pseudocompact.
Any function $f\in C^*(X\times Y)$ extends to a continuous bounded function on $\bt X\times Y$. From \cite[Theorem 1.1]{ComfortHager1971} it follows that the projection $\pi$ is closed.

(2) $\rarr$ (1)
Let $f\in C(X\times Y)$. Let $\ph\: \R\to (0,1)$ be some homeomorphism, and let $h=\ph\circ f$. Then $h\in C^*(X\times Y)$ and the projection $\pi$ is $z$-closed. From \cite[Theorem 1.1]{ComfortHager1971} it follows that the function $h$ extends to a continuous bounded function $\hat h\in C^*(\bt X\times Y)$. Since $X$ is pseudocompact, we have $\hat h(\bt X\times Y)\subset (0,1)$. Let us set $\hat f = \ph^{-1}\circ \hat h$.
Then $\hat f$ is a continuous extension of $f$.
\end{proof}

The following statement follows from Proposition \ref{p:ext:2}.

\begin{proposition}\label{p:ext:3}
A space $Y$ is $C$-universal if and only if the projection $\pi\: X\times Y\to Y$ is $z$-closed for any pseudocompact space $X$.
\end{proposition}

Let $Y$ be a space. Let us denote by $\cS$ all subsets $Y'\subset Y$ for which $X\times Y'$ is bounded in $X\times Y$ for any pseudocompact space $X$.
The space $Y$ is called a \term{$b^*_R$-space} \cite{Noble1969} if a function $f\: Y\to\R$ is continuous whenever its restriction to each subset $Y'\in \cS$ is continuous.

The following statement follows from Proposition \ref{p:ext:3} and \cite[Corollary 1]{Noble1969}.

\begin{proposition}\label{p:ext:4}
Any $b^*_R$-space is $C$-universal.
\end{proposition}

The product of a pseudocompact space and a compact space is pseudocompact \cite[Corollary 3.10.27]{EngelkingBookGT}, so $\cS$ contains all compact subsets of $Y$. Therefore, any $k$-space is a $b^*_R$-space.

\begin{cor}\label{c:ext:1}
Any $k$-space is $C$-universal.
\end{cor}

\begin{proposition}\label{p:ext:5}
Let $T$ and $Y$ be spaces, let $X_1$, $X_2$, ..., $X_n$ be pseudocompact spaces,
and let $\F: T\times \prod_{i=1}^n X_n \to Y$ be a 
separately continuous mapping. Suppose that $(X_i,T)$ is Grotendieck for $i=1,2,...,n$.
\begin{enumerate}
\item
If the mapping $\F$ extends to a mapping $\Psi\: T\times \prod_{i=1}^n \bt X_n \to \bt Y$ such that the mapping
\[
\Psi_t\: \prod_{i=1}^n \bt X_n \to \bt Y,\ (x_1,x_2,...,x_n) \mapsto \Psi(t,x_1,x_2,...,x_n)
\]
is separately continuous for  each $t\in T$, then the mapping $\Psi$ is separately continuous.
\item
If the mapping
\[
\F_t\: \prod_{i=1}^n X_n \to Y,\ (x_1,x_2,...,x_n) \mapsto \F(t,x_1,x_2,...,x_n)
\]
is $2$-quasicontinuous for any $t\in T$, then the mapping $\F$ extends to a separately continuous mapping $\Psi\: T\times \prod_{i=1}^n \bt X_n \to \bt Y$.
\end{enumerate}
\end{proposition}
\begin{proof}
Let us prove (1). Consider the case when $Y=[0,1]$. We put
\[
\ph=\diag_{t\in T} \Psi_t\: \prod_{i=1}^n \bt X_n \to [0,1]^T,\ \ph(x_1,x_2,...,x_n) (t)=\Psi_t(x_1,x_2,...,x_n).
\]
For $i=0,1,...,n$, we set $P_i=\bt X_1\times \bt X_2\times ... \bt X_{i}\times X_{i+1} \times ... \times X_n$ and $\ph_i = \restr\ph{P_i}$.

The mapping $\ph$ is separately continuous and $\ph(P_0)\subset C_p(T,[0,1])$.
Let us show that $\ph(P_i)\subset C_p(T,[0,1])$ for $i=1,2,...,n$. Suppose that $\ph(P_{i-1})\subset C_p(T,[0,1])$.
Let $(x_1,x_2,...,x_n)\in P_i$. 
We denote 
\[
f\: \bt X_i\to [0,1]^T,\ x\mapsto \ph(x_1,x_2,...,x_{i-1},x,x_{i+1} ,...,x_n). 
\]
The mapping $f$ is continuous and $f(X_i)\subset C_p(T,[0,1])$.
Since $(X_i,T)$ is Grotendieck, it follows that $\cl{f(X_i)}$ is a compact subset of $C_p(T,[0,1])$. Therefore, $f(\be X_i)\subset C_p(T,[0,1])$.
Hence $\ph(P_n)\subset C_p(T,[0,1])$. Consequently, the mapping $\Psi$ is separately continuous.
%
%

Let us consider the general case. We denote $\cF=C(Y,[0,1])$. Let $f\in \cF$ and $t\in T$.
Let $\hat f\: \bt Y\to [0,1]$ be a continuous extension of the function $f$.
We put $\F_f=f\circ \F$, $\Theta_f= \hat f \circ \Psi$, and $\Theta_{f,t}=\hat f \circ \Psi_t$.
Since the mapping $\Psi_t$ is separately continuous, it follows that so is the mapping $\Theta_{f,t}$. We have
\[
 \Theta_{f,t}(x_1,x_2,...,x_n)= \Theta_{f}(t,x_1,x_2,...,x_n)
\]
for $x_1,x_2,...,x_n\in X$; hence  assertion (1) with $Y=[0,1]$ implies that the mapping $\Theta_{f}$ is separately continuous.
We put
\begin{align*}
\Theta &= \diag_{f\in \cF} \Theta_f\: T\times \prod_{i=1}^n \bt X_n \to [0,1]^{\cF},
&
\delta &= \diag_{f\in \cF} \hat f\: \bt Y\to [0,1]^{\cF}.
\end{align*}
Then the map $\Theta$ is separately continuous and $\Theta = \Psi \circ \delta$. Since $\delta$ is a homeomorphic embedding of $\bt Y$ into $[0,1]^{\cF}$, it follows that the map $\F$ is separately continuous.

Let us prove (2).
From \cite[Theorem 5]{Reznichenko2022-2} it follows that the mapping $\F_t$ extends to a separately continuous mapping $\Psi_t\: \prod_{i=1}^n \bt X_n \to \bt Y$. Let us put $\Psi(t,x_1,x_2,...,x_n)=\Psi_t(x_1,x_2,...,x_n)$. It follows from (1) that the mapping $\Psi$ is separately continuous.
\end{proof}

\begin{proof}[Proof of Theorem \ref{t:main:7}]
Let $n\in \sp E$. From Proposition \ref{p:ext:5} it follows that the mapping $\Psi_{n}$ extends to a separately continuous mapping $\wh \Psi_n\: E_n\times (\bt X)^n\to \bt X$.
\end{proof}

\begin{proof}[Proof of Theorem \ref{t:main:8}]
Let $n\in \sp E$ and $f\in E_n$. From \cite[Theorem 6]{Reznichenko2022-2} and \cite[Theorem 5]{Reznichenko2022-2} it follows that the mapping $\Psi_{n,f}$ is $2$-quasicontinuous. Next we apply Theorem \ref{t:main:7}.
\end{proof}


\section{Identities in extensions of $E$-algebras}
\label{sec:ident}
Let us prove Theorem \ref{t:main:11}.
Let $R=\vars p\cup \vars q$. We put
\begin{align*}
\vt p\: &\, (\bt X)^R\to \bt X,\ \bar x\mapsto p^{\bt }(\bar x),
&
\vt q\: &\, (\bt X)^R\to \bt X,\ \bar x\mapsto q^{\bt }(\bar x).
\end{align*}
We must prove that the mappings $\vt p$ and $\vt q$ coincide.
Let $s(\bar x)=|\set{i\in\sset{1,2,...,n}: x_i\notin X}|$. By induction on $s(\bar x)$ we show that $\vt p(\bar x)=\vt q(\bar x)$.

{\bf Base case.} If $s(\bar x)=0$, then $\bar x\in X^R$. Since the mappings $\vt p$ and $\vt q$ coincide on $X^R$, we have $\vt p(\bar x)=\vt q(\bar x)$.

{\bf Induction step.} Let $m=s(\bar x)>0$, and let $\vt p(\bar y)=\vt q(\bar y)$ for $\bar y\in (\bt X)^R$ with  $s(\bar y)< m$. Let $v\in R$ and $\bar x(v)\notin X$.
We define $\xi\: \bt X\to (\bt X)^R$ by setting $\xi(x)(u)=\bar x(u)$ if $u\neq v$ and $\xi(x) (u)=x$ if $u=v$. We put
\begin{align*}
\bar p\: &\, \bt X\to \bt X,\ x\mapsto \vt p(\xi(x)),
&
\bar q\: &\, \bt X\to \bt X,\ x\mapsto \vt q(\xi(x)).
\end{align*}
To prove the theorem, it is enough to check that the functions $\bar p$ and $\bar q$ coincide.
The induction hypothesis implies that $\bar p(x)=\bar q(x)$ for $x\in X$.
For $p$ and $q$, one of the conditions (1) and (2) in the statement of the  theorem is satisfied. Suppose for definiteness that this is condition (2): $\Om(p)(v)\leq 2$ and $\Om(q)(v)\leq 1$. Since $\Om(q)(v)\leq 1$, it follows that the function $\bar q$ is continuous.
If $\Om(p)(v)\leq 1$, then the function $\bar q$ is also continuous, $\bar p$ and $\bar q$ coincide on the dense set $X$, and hence $\bar p= \bar q$.

Consider the case $\Om(p)(v)= 2$. The variable $v$ twice occurs in the term $p$. We replace one of the occurrences with some variable $u\in \vV\setminus R$ and denote the resulting term by $r$. 
Let us define $\ph\: (\bt X)^2\to \bt X$  by setting $\ph(x,y)=r^{\bt X}(\bar z_{x,y})$, where $\bar z_{x,y}\in X^{R\cup \sset u}$ is defined by
\[
\bar z_{x,y}(w) = \begin{cases}
x&  \text{if } w=v,
\\
y & \text{if } w= u,
\\
\bar x(w) & \text{otherwise}
\end{cases}
\]
for $w\in R\cup \sset u$.
%
Then $\bar p(x)=\ph(x,x)$ and the function $\ph$ is separately continuous.
Since $\bar q(x)=\bar p(x)=\ph(x,x)$ for $x\in X$ and $\bar q$ is continuous, it follows that $\ph$ is continuous on $ \D_X=\set{(x,x):x\in X}$. Since the space $X$ is pseudocompact, the map $\ph$ is separately continuous, and $\ph$ is continuous on $\D_X$, it follows from \cite[Proposition 3.12]{ReznichenkoUspenskij1998} that $\ph$ is continuous on $\D_{ \bt X}=\set{(x,x):x\in \bt X}$, that is, the function $\bar p(x)=\ph(x,x)$ is continuous. Thus, the continuous functions $\bar p$ and $\bar q$ coincide on the dense set $X$. Therefore, $\bar p=\bar q$.


\section{Factorization of mappings and embedding of algebras into  products of separable metrizable algebras}
\label{sec:rfmap}

It is easy to see that a product  space $\prod_{i=1}^n X_i$
   is (separately) $\R$-factorizable over a space $T$ if, for any separable metrizable space $Z$ and any (separately) continuous function $f\: T\times \prod_{i=1}^n X_i\to Z $, there are separable metrizable spaces $Y_1$, $Y_2$, ..., $Y_n$, continuous maps $h_i\: X_i\to Y_i$ for $i=1, 2, ..., n$, and a continuous map $g \: T\times \prod_{i=1}^n Y_i\to Z$ such that $f=g\circ h$, where $h=\id_T\times \prod_{i=1}^n h_i$.
 
If all spaces $X_i$ coincide with each other and $X=X_i$ for $i=1,2,...,n$, then, passing to $Y=\prod_{i=1}^n Y_i$, we obtain a criterion for the (separate) $\R$-factorizability of a product $X^n$ over a space $T$: for any separable metrizable space $Z$ and any (separately) continuous function $f\: T\times X^n\to Z$, there is a separable metrizable space $Y$ and continuous maps $h\: X\to Y$ and $g\: T\times Y^n\to Z$ such that $f=g\circ (\id_T\times h^n) $.

\begin{proof}[Proof of Theorems \ref{t:main:12} and \ref{t:main:13}]
Let $f\: X\to\R$ be a continuous function.
We set $Z_1=f(X)$ and $f_1=f\: X\to Z_1$. Take a one-point space $Z_0$ and  let  $f_0\: X\to Z_0$, $g_0^1\: Z_1\to Z_0$, and $p_ {0,n}\: E_n\times Z_{1}^n\to Z_0$ be constant mappings.

By induction on $m\in\om$, we construct separable metrizable  spaces $Z_m$, continuous surjective mappings $f_m\: X\to Z_m$ and $g_m^{m+1}\: Z_{m+1}\to Z_m$, and (separately) continuous mappings $p_{m,n}\: E_n\times Z_{m+1}^n\to Z_m$ such that, for $\te_{n,m}=\id_{E_n}\times ( g_m^{m+1})^n$ and $\nu_{n,m}=\id_{E_n}\times f_m^n$, the following diagram is commutative:
\[
\begin{tikzcd}
&&&& X
\ar[lllld,"f_0"']
\ar[llld,"f_1"]
\ar[ld,"f_m"]
\\
Z_0 & Z_1 \ar[l,"g_0^{1}"] & ... \ar[l,"g_1^{2}"] & Z_{m} \ar[l,"g_{m-1} ^{m}"] & ... \ar[l,"g_{m}^{m+1}"]
&
\\
E_n\times Z_{1}^n \ar[u,"p_{0,n}"] & E_n\times Z_{2}^n \ar[u,"p_{1,n}"] \ar[ l,"\te_{n,2}"'] & ... \ar[l,"\te_{n,3}"'] & E_n\times Z_{m+1}^n \ar[u, "p_{m,n}"'] \ar[l,"\te_{n,m+1}"'] & ... \ar[l,"\te_{n,m+2}"']
&
\\
&&&& E_n\times X^n
\ar[uuu, bend right, "\Psi_n"]
\ar[lu,"\nu_{n,m+1}"']
\ar[lllu,"\nu_{n,2}"']
\ar[llllu,"\nu_{n,1}"]
\end{tikzcd}
\]

Let $m>1$. Let us assume that metrizable separable spaces $Z_{m-1}$, continuous surjective mappings $f_{m-1}\: X\to Z_{m-1}$, 
$g_{m-2}^{m-1}\: Z_{m-1}\to Z_{m-2}$, and (separately) continuous mappings $p_{m-2,n}\: E_n\times Z_{ m-1}^n\to Z_{m-2}$ such that $f_{m-2}=g_{m-2}^{m-1}\circ f_{m-1}$ and $ f_{m-2} \circ \Psi_n = p_{m-2,n} \circ (\id_{E_n} \times f_{m-1}^n)$ for $n\in\sp E$  are defined.

Let $n\in\sp E$. Let us set $\ph_{m,n}=f_{m-1}\circ \Psi_{n}\: E_n\times X^n\to Z_{n-1}$.
Since $X^n$ is (separately) $\R$-factorizable over $E_n$, it follows that there is a separable metrizable space $Z_{m,n}$, continuous maps $h_{m,n}\: X\to Z_{ m,n}$, and a (separately)  continuous map $r_{m,n}\: E_n\times Z_{m,n}^n\to Z_{m-1}$ such that $\ph_{m,n} =r_{m,n}\circ (\id_{E_n}\times h_{m,n}^n)$.

Let us put
\[
f_m=f_{m-1}\diagsmall \diag_{n\in \sp E}h_{m,n}\: X\to \Pi=Z_{m-1} \times \prod_{n\in \sp E } Z_{n,m}
\]
and $Z_m=f_m(X)$. 
Let $\pi_{m,n}: \Pi\supset Z_n\to Z_{n,m}$ and $g^m_{m-1}\: \Pi\supset Z_{m}\to Z_{m-1 }$ be projections.
Let us put $p_{m-1,n}=r_{m,n} \circ (\id_{E_n}\times \pi_{m,n}^n) \: E_n\times Z_m^n\to Z_{m- 1}$.
The construction is completed.

We set $h=\diag_{m\in\om}f_m\: X\to \prod_{m\in\om}Z_m$, $Y=h(X)$,
\[
\op Yn
=\rlim\sset{p_{n,m},\om}\: E_n\times Y^n=\rlim\sset{E_n\times Z_m,\te_{n,m},\om}\to Y =\rlim\sset{Z_m,g_m^{m+1},\om}
\]
for $n\in\sp E$ and let $g\: Y\to Z_1\subset \R$ be the projection from $Y$ to $Z_1$.
Then $Y$ is a (semi)topological algebra with   algebraic structure $\sqn{\op Yn}$, the function $g$ is continuous, $h\: X\to Y$ is a continuous homomorphism, and $f=g\circ h$.
\end{proof}

A product  space $\prod_{i=1}^n X_i$ is called \term{$\R$-factorizable} if the product $\prod_{i=1}^n X_i$ is $\R$-factorizable over a one-point space \cite{Reznichenko2024rf}.
We call a subset $U\subset \prod_{i=1}^n X_n$ \term{rectangular} if $U=\prod_{i=1}^n U_n$ for some $U_i\subset X_i$,  $i= 1,2,...,n$.

The following statement is easy to prove.

\begin{proposition}\label{p:rfmap:3-1}
A product  space $P=\prod_{i=1}^n X_i$ is $\R$-factorizable if and only if any cozero set $W\subset P$ is a countable union of rectangular cozero sets.
\end{proposition}

Since  any cozero set in a Lindelöf space  is Lindelöf, we obtain the following corollary of Proposition \ref{p:rfmap:3-1}.

\begin{proposition}\label{p:rfmap:3}
If a  product  space $P=\prod_{i=1}^n X_i$ is Lindelöf, then $P$ is $\R$-factorizable.
\end{proposition}

Given spaces $X_1$, $X_2$, ..., $X_n$ and a space $Y$, we denote by $SC_p(\prod_{i=1}^n X_i,Y)$ the space of separately continuous mappings from $P=\prod_{i=1}^n X_i$ to $Y$ in the topology of pointwise convergence, that is, in the topology inherited from $Y^P$. We set $SC_p(\prod_{i=1}^n X_i,\R)=SC_p(\prod_{i=1}^n X_i)$.

\begin{proposition}\label{p:rfmap:4}
Let $X_1$, $X_2$, ..., $X_n$ be compact spaces with caliber $\om_1$, and let $T$ be a separable space.
Then any compact subset of $SC_p(T\times\prod_{i=1}^n X_i)$ is metrizable.
\end{proposition}
\begin{proof}
Let $\set{t_n:\nom}$ be a dense subset of $T$.
Let us denote $Y=SC_p(T\times\prod_{i=1}^n X_i)$, $Z=SC_p(\prod_{i=1}^n X_i)$, and
\[
\ph_m\: Y \to Z,\ \ph_m(\F)(x_1,x_2,...,x_n)=\F(t_m,x_1,x_2,...,x_n)
\]
for $m\in\om$.
The compact subsets of $Z$ are metrizable \cite[Proposition 3.20]{ReznichenkoUspenskij1998}.
Since the mapping $\ph = \diag_{m\in\om} \ph_m\: Y\to Z^\om$ is injective,  it follows that the compact subsets of $Y$ are metrizable.
\end{proof}

\begin{proof}[Proof of Proposition \ref{p:main:3}]
Let us prove (1). From Proposition \ref{p:rfmap:3} it follows that the product $T\times \prod_{i=1}^n X_i$ is $\R$-factorizable. Consequently, the product $\prod_{i=1}^n X_i$ is $\R$-factorizable over $T$.

Let us prove (2). Let $f\: T\times \prod_{i=1}^n X_i\to \R$ be a separately continuous map.
Let us put
\begin{align*}
\de_i\: X_i & \to SC_p(T\times \prod_{j\neq i} X_j),
\\
& h_i(x)(x_1,x_2,...,x_{i-1},x_{i+1},...,x_n)=\F(x_1,x_2,...,x_{i- 1},x,x_{i+1},...,x_n)
\end{align*}
for $i=1,2,...,n$. From Proposition \ref{p:rfmap:4} it follows that the spaces $Y_i=h_i(X_i)$ are compact metrizable spaces for $i=1,2,...,n$.
From \cite[Claim 3.9]{ReznichenkoUspenskij1998} it follows that there is a separately continuous function
$g\: T\times \prod_{i=1}^n Y_i$ such that $f=g\circ h$, where $h=\id_T\times \prod_{i=1}^n h_i$.
\end{proof}


\section{Least factorization}
\label{sec:mfac}
In this section we do not assume any  separation axioms.
Let us prove Theorem \ref{t:main:15}.

Let us put
\[
g=\diag_{\si\in S(X)} f \circ \si \: X\to Z^{S(X)},\ \ \ Y=g(X).
\]
For $\si\in S(X)$, we denote by $\pi_\si\: Y\to Z$ the $\si$-projection $Z^{S(X)}\supset Y$: $\pi_\si(y)=y(\si)$.
Let us denote $h=\pi_e$. By construction, $Y\subset Z^{S(X)}$, the mappings $g$ and $h$ are continuous, and $f=h\circ g$.
To prove the theorem, it remains to prove that on $Y$ there is a structure $\sqn{\op Yn}$ of a semitopological $E$-algebra with respect to which the mapping $g$ is a homomorphism, and the factorization $(g,Y, h)$ is least.

For $x,y\in X$, we set $x\sim y$ if $g(x)=g(y)$. Let us check that the equivalence relation $\sim$ is a congruence. To do this, it is enough to check that if $x\sim y$ and $\de\in S_1(X)$, then $\de(x)\sim \de(y)$ (see Proposition \ref{p:prel:1}). Note that $\de(x)\sim \de(y)$ if and only if $g(\de(x))=g(\de(y))$ or, equivalently,
\begin{equation}\label{eq:mfac:1}
\pi_\si(g(\de(x)))=\pi_\si(g(\de(y)))
\end{equation}
for all $\si\in S(X)$; thus, it suffices to check that (\ref{eq:mfac:1}) holds for $\si\in S(X)$.
The relations  $\pi_\si(g(\de(x)))=f(\si(\de(x)))=f((\si\circ\de)(x))=\pi_{\si\circ \de}(g(x))$, $\pi_\si(g(\de(y)))=\pi_{\si\circ \de}(g(x))$, and $g (x)=g(y)$  imply (\ref{eq:mfac:1}).
Hence $\sim$ is a congruence and there is an $E$-algebra structure $\sqn{\op Yn}$ on $Y$ such that the mapping $g$ from $X$ to $Y$ is a homomorphism.
In order to verify that $Y$ is a semitopological algebra, it is necessary and sufficient to check the following conditions:
\begin{enumerate}
\item[($SC_1$)] any principal translation $\te\in S_1(Y)$ is continuous;
\item[($SC_2$)] for $n\in\sp E$ and $\bar y=(y_1,y_2,...,y_n)\in Y^n$, the mapping $\ph\: E_n\to Y $, $f\mapsto \op Yn(f,\bar y)$, is continuous.
\end{enumerate}

Let us check ($SC_1$). There is $\de\in S_1(X)$  such that $g\circ \de=\te\circ g$.
Let $y\in Y$. Then $y=g(x)$ for some $x\in X$. Since
$\pi_\si(\te(y))=\pi_\si(g(\de(x)))=\pi_{\si\circ \de}(g(x))=\pi_{\si \circ \de}(y)$ and the mapping $\pi_{\si\circ \de}$ is continuous, it follows that $\pi_\si\circ \te$ is continuous for all $\si\in S(X) $. Therefore $\te$ is continuous.

Let us check ($SC_2$). There is $\bar x\in X^n$  such that $\bar y= g^n(\bar x)$. We have $\ph=g\circ \psi$, where 
\[
\psi\: E_n\to Y,\ f\mapsto \Psi_n(f,\bar x). 
\]
Since the mappings $g$ and $\psi$ are continuous, it follows that $\ph$ is continuous.

Let us check that the factorization $(g,Y, h)$ is least. Let $(q,S,r)$ be a continuous factorization of the map $f$, where $S$ is a semitopological $E$-algebra. The homomorphism $q$ of algebras induces a homomorphism $\vt q\: S(X)\to S(S)$ of translation semigroups such that $\vt q(\si)\circ q = q\circ \si $ for $\si\in S(X)$. Let us put
\[
s=\diag_{\si\in S(X)} r \circ \vt q(\si) \: S\to Z^{S(X)}.
\]
Since $f \circ \si=r\circ q\circ \si =r \circ \vt q(\si) \circ q$, we have $g=s\circ q$.
Therefore $(g,Y, h)\prec (q,S,r)$.


\bibliographystyle{elsarticle-num}
\bibliography{ualg}
\end{document}